\documentclass[12pt,english]{article}
\usepackage{amsmath,amssymb,amsfonts,amsthm,amscd}
\usepackage{enumerate}
\usepackage{a4wide}
\usepackage{textcomp}
\usepackage{authblk}
\usepackage{fix-cm}
\usepackage{lmodern}
\usepackage{url}   
\newtheorem{theorem}{Theorem}[section]

\newtheorem{lemma}[theorem]{Lemma}
\newtheorem{remark}[theorem]{Remark}

\newtheorem{algorithm}{Algorithm}[section]

\def\R{\mathbb{R}}

\newcommand\BibTeX{{\rmfamily B\kern-.05em \textsc{i\kern-.025em b}\kern-.08em
T\kern-.1667em\lower.7ex\hbox{E}\kern-.125emX}}

\begin{document}

\title{Analysis of the mean squared derivative cost function}

\author[1]{Manh Hong Duong}
\author[2,3] {Minh Hoang Tran}
\affil[1]{Mathematics Institutes, University of Warwick, Coventry CV4 7AL. Email: m.h.duong@warwick.ac.uk.}
\affil[2]{Department of Industrial Engineering, Texas A\&M University, College Station, TX 77843. Email: tran@tamu.edu}
\affil[3]{School of Applied Mathematics and Informatics, Hanoi University of Science \& Technology, Hanoi, Vietnam}
\maketitle

\begin{abstract}
In this paper, we investigate the mean squared derivative cost functions
that arise in various applications such as in motor control, biometrics and optimal transport theory. We provide qualitative
properties, explicit analytical formulas and computational algorithms for the
cost functions. We also perform numerical simulations to illustrate the
analytical results. In addition, as a by-product of our analysis, we obtain an explicit formula for the inverse of a Wronskian matrix that is of independent interest in linear algebra and differential equations theory.
\end{abstract}
\section{Introduction}
\subsection{The mean squared derivative cost function}

This paper is concerned with the analysis of \textit{the mean squared derivative cost function} defined as follows: 
\begin{equation}
\label{eq: Ch cost}
C_{n,h}(x_{0},x_{1},\ldots,x_{n-1};y_{0},y_{1},\ldots,y_{n-1}):=\inf\limits_{\xi}\int_0^h|{\xi}^{(n)}(t)|^{2}\,dt,
\end{equation}
where $\mathbf{x}=(x_{0},\ldots,x_{n-1})\in\R^{dn},\mathbf{y}=(y_{0},\ldots,y_{n-1})\in\R^{dn}$,
the infimum is taken over all curves $\xi\in C^{n}([0,h],\R^d)$ that satisfy the boundary conditions
\begin{equation}
\label{eq: boundary conditions}
(\xi,\xi',\ldots,\xi^{(n-1)})(0)=(x_{0},x_{1},\ldots,x_{n-1})\quad\text{and}\quad (\xi,\xi',\ldots,\xi^{(n-1)})(h)=(y_{0},y_{1},\ldots,y_{n-1}).
\end{equation}
Throughout this paper, $h$ is a constant representing the final time.

This cost function plays a central role in various practical applications and theoretical research. For the motivation of this paper, we now review some important literature here.
\subsection{Literature review and motivation of the present paper}
\textit{Applications in engineering and applied sciences}. In the literature, the minimization problem \eqref{eq: Ch cost} for $n=2,3,4,5,6$ is respectively called the principle of minimum acceleration, jerk,  snap, crackle, and pop; see for instance \cite{RichardsonFlash02}. The minimal jerk principle was initially used to model velocity profiles generated by elbow movements \cite{Hogan84a} and later extended to trajectory prediction for reaching movements between visual targets in the horizontal plane and to curved and obstacle-avoidance movements \cite{FlashHogans85}. Since then the minimal jerk principle and other mean squared derivative cost functions \eqref{eq: Ch cost} have been found to be useful in the modelling and design of various real-world systems. Examples of such applications include motor control \cite{Engelbrecht2001}, biometrics and online-signatures \cite{coutinhocanutoPhd, CanutoDorizziMontalvao2013} and robotics \cite{Freeman2012}, just to name a few. We refer to the mentioned papers and references therein for more information.  Since the cost function has applications in many different contexts, a thorough analysis and computational method for the general case, i.e. with arbitrary $n$ and boundary values $\mathbf{x}$ and $\mathbf{y}$, would be useful. For instance, it can be pre-computed and embedded in a larger algorithm. However, in the literature there exist neither analytic formulas nor computational methods for the general case.
\begin{center}
\textbf{Question 1.} \textit{Can one establish an explicit formula and an algorithm to compute the cost function $C_{n,h}$ for the general case?}
\end{center}

\textit{Usage in theoretical research}. The mean squared derivative cost functions also appear in theoretical research such as in partial differential equations theory and optimal transport theory. When $n=1$, the function
\begin{equation*}
\Phi_1(t,x,y)=\frac{1}{(4\pi t)^{d/2}}\exp\left(-\frac{C_{1,t}(x,y)}{4}\right)
\end{equation*}
is the fundamental solution of the diffusion/heat equation $\partial_t u=\Delta u$. When $n=2$, the function
\begin{equation*}
\Phi_2(t,x,x';y,y')=\frac{\beta_d}{t^{2d}}\exp\left(-\frac{C_{2,t}(x,x';y,y')}{4}\right),
\end{equation*}
where $\beta_d$ is a normalising constant, is the fundamental solution of the following ultra-parabolic equation
\begin{equation}
\label{eq: ultra-pde}
\partial_t u+y\cdot\nabla_x u=\Delta_y u.
\end{equation}
where subscripts in $\nabla_x$ and $\Delta_y$ indicate that these differential operators act only on those variables. This equation was first studied by Kolmogorov \cite{Kol34} and then was H\"{o}rmander's starting point to develop the hypo-elliptic theory \cite{Hormander67}.

The diffusion/heat equation and the above ultra-parabolic equation are special cases of the following hypo-elliptic equation
\begin{equation}
\label{eq: hypo-eqn}
\partial_t u+\sum_{i=1}^{n-1} x_{i+1}\cdot \nabla_{x_i}u=\Delta_{x_n}u.
\end{equation}
Equations of this type have been studied from various points of view such as trends to equilibrium \cite{DV2001}, Gaussian-estimates for the fundamental solution \cite{DelarueMenozzi10},  connections to particle systems and coarse-graining \cite{DPZ13b, Duong15NA, DLPS2015TMP}. 

A natural question arises

\textbf{Question 2.} \textit{Is 
\begin{equation}
\label{eq: Phi}
\Phi_n(t,\mathbf{x},\mathbf{y}):=\frac{\tilde{\beta}_d}{t^{\frac{n^2 d}{2}}}\exp\Big(-\frac{C_{n,t}(\mathbf{x},\mathbf{y})}{4}\Big), \quad\text{for some normalising constant}~~\tilde{\beta_d},
\end{equation}
the fundamental solution of \eqref{eq: hypo-eqn}?
}

In the optimal transport theory, the cost function $C_{n,h}$ between two points $\mathbf{x}$ and $\mathbf{y}$ in the Euclidean space $\R^{dn}$ can be used to define a Monge-Kantorovich optimal transport cost function between two probability measures $\mu(d\mathbf{x})$ and $\nu(d\mathbf{y})$ on $\R^{dn}$ as follows
\begin{equation}
\label{eq: MK cost}
\mathcal{W}_{n}^2(\mu,\nu):=\inf_{\gamma\in \Gamma(\mu,\nu)}\int_{\R^{dn}\times \R^{dn}}C_{n,h}(\mathbf{x};\mathbf{y})\,\gamma(d\mathbf{x}d\mathbf{y}),
\end{equation}
where $\Gamma(\mu,\nu)$ denotes the set of all probability measures on $\R^{dn}\times\R^{dn}$ having $\mu$ and $\nu$ as first and second marginals.
The Monge-Kantorovich cost function is the central object in optimal transport theory with many applications in other fields of mathematics and economics; see for instance \cite{Vil03,Vil09} for nice expositions on optimal transport theory and its applications. In particular, for $n=1$, $C_{1,h}(x_0,y_0)=\frac{1}{h}|y_0-x_0|^2$ and $\mathcal{W}_{1}$ is the well-known Wasserstein distance. Wasserstein gradient flows, i.e., gradient flows of energy functionals with respect to the Wasserstein metric, form an important class of dissipative evolution equations, see for instance \cite{JKO98,AGS08}. When $n=2$, $C_{2,h}(x_0,x_1;y_0,y_1)=\frac{1}{h}\Big[|y_1-y_0|^2+12\big|\frac{x_1-x_0}{h}-\frac{y_1+y_0}{2}\big|^2\Big]$ and $\mathcal{W}_{2}$ is the minimal acceleration cost function. This cost function has been used to construct variational formulation for the Kramers equation (Equation \eqref{eq: ultra-pde} above with additional terms coming from external and frictional forces) showing that the Kramer equation is a (generalised) gradient flow of the Boltzmann entropy with respect to the  Monge-Kantorovich transport cost $\mathcal{W}_{2}$  \cite{Hua00,DPZ13a}. In addition, $\mathcal{W}_{2}$ has also been used in constructing variational schemes for other evolution equations such as the system of isentropic Euler equations \cite{GW09, Westdickenberg10} and the compressible Euler equations \cite{CSW2014TMP}. 

\textbf{Question 3}. \textit{Is Equation \eqref{eq: hypo-eqn} a (generalised) gradient flow of the Boltzmann entropy with respect to the Monge-Kantorovich transport cost $\mathcal{W}_{n}$?}

We provide further discussions on the motivation of the present paper in Section \ref{sec: futher dis}.

\subsection{The aim of the present paper}

The aim of this paper is to address Question 1. We develop analytic and computational aspects. We show some qualitative properties of $C_{n,h}$ in Theorem \ref{theo: qual theo}; we provide an explicit analytical formula for $C_{n,h}$ in Theorem \ref{theo: explicit formula}; and we present a computational method for $C_{n,h}$ in Theorem \ref{theo: LU of A} and in Algorithm \ref{al: alg}.

Questions 2 and 3 will be answered in a companion paper \cite{DuongTran2017} where, using analytical formulas in Theorem  \ref{theo: explicit formula} and Theorem \ref{theo: LU of A} of  the present paper, we will prove that the function $\Phi$ defined in \eqref{eq: Phi} is the fundamental solution of Equation \eqref{eq: hypo-eqn} and show that this equation is indeed a (generalised) gradient flow of the Boltzmann entropy with respect to the Monge-Kantorovich transport cost $\mathcal{W}_{n}$ via a variational approximation scheme.
\subsection{Main results of the present paper}
We now describe our main results. The detailed statements will be given in the subsequent sections. Our first result concerns the qualitative behaviour of the cost function as function of $h$ and $n$. 
\begin{theorem}
\label{theo: qual theo}
The cost function $C_{n,h}$ is scalable with respect to $h$ and monotonically increasing with respect to $n$.
\end{theorem}
The full description and proof of this theorem are given in Theorem \ref{theo: qual theo 2} in Section \ref{sec: qualitative}. Moreover, we also provide an interpretation of the cost function based on the theory of large deviations.

Our second theorem is an explicit analytical formula for the cost function.
\begin{theorem}
\label{theo: explicit formula}
The cost function $C_{n,h}$ has an explicit formula given by
\begin{equation}
C_{n,h}(x_0,\ldots,x_{n-1};y_0,\ldots,y_{n-1})=\mathbf{b}_n(h)^T B_n(h)[A_n(h)]^{-1}\mathbf{b}_n(h),\label{costfun}
\end{equation}
where the vector $\mathbf{b}_n(h)$ and the two matrices $B_n(h)$ and $A_{n}(h)$ are given explicitly in \eqref{b}, \eqref{B} and \eqref{A} respectively.
\end{theorem}
The proof of this theorem is given in Section \ref{sec: formula}.

The last theorem provides explicit formulas for the $LU$ decomposition of $A_n$ and for $A_n^{-1}$.
\begin{theorem}
\label{theo: LU of A}
The matrix $A_n(h)$ has an $LU$-decomposition as in \eqref{U}-\eqref{L}. The inverses of the matrices $L$ and $U$ are given explicitly in \eqref{heq3}-\eqref{eq21}.
\end{theorem}
We prove this theorem in Section \ref{sec: LU} (cf. Theorem \ref{theo: LU of A 2}). As we show there, the matrix $A_n(h)$ is a Wronskian matrix, which plays an important role in linear algebra and differential equations; hence, this theorem is of independent interest.

\subsection{Organisation of the paper}
The rest of this paper is structured as follows. In Section \ref{sec: qualitative}, we study some qualitative properties of the cost function $C_{n,h}$. In Section \ref{sec: formula} we provide an explicit formula for $C_{n,h}$. The $LU$-decomposition of the matrix $A$ is presented in Section \ref{sec: LU}.  In Section \ref{sec: simulations} we provide an algorithm to compute the cost function and compute the expressions obtained for small $n$ explicitly.
\section{Qualitative properties of the cost functions}
\label{sec: qualitative}
In this section, we provide the full description and proof of Theorem \ref{theo: qual theo} on  qualitative properties of the cost function, which characterises
its behaviour as functions of $n$ and $h$. 
\subsection{Behaviour of $C_{n,h}$ as a function of $n$ and $h$}
\begin{theorem}[Qualitative properties of the cost function]\
\label{theo: qual theo 2}
\begin{enumerate}
\item (Scaling
property of the cost functions). It holds that
\begin{equation}
C_{n,h}(x_{0},x_{1},\ldots,x_{n-1};y_{0},y_{1},\ldots,y_{n-1})= h^{1-2n}\inf_{\xi}\,\int_{0}^{1}|{\xi}^{(n)}(t)|^{2}\,dt,\label{eq: Ch1}
\end{equation}
where the infimum is taken over the curves $\xi\in C^{n}([0,1],\R^{d})$ such that
\begin{equation*}
(\xi,\xi',\ldots,\xi^{(n-1)})(0)=(x_{0},hx_{1},\ldots,h^{n-1}x_{n-1}),\quad(\xi,\xi',\ldots,\xi^{(n-1)})(1)=(y_{0},hy_{1},\ldots,h^{n-1}y_{n-1}).
\end{equation*}
\item (Monotonicity of the cost function). It holds that  
\begin{equation}
\label{monotone}
C_{n-1,h}(x_{1},\ldots,x_{n-1};y_{1},\ldots,y_{n-1})\leq C_{n,h}(x_{0},x_{1},\ldots,x_{n-1};y_{0},y_{1},\ldots,y_{n-1}).
\end{equation}
\end{enumerate}
\end{theorem} 
Note the difference between the right-hand sides of \eqref{eq: Ch cost}
and \eqref{eq: Ch1}. In the former, the dependence on $h$
appears in the interval of the integral, while in the latter, the
dependence is moved to the boundary conditions. The pre-factor is
also rescaled accordingly. 
\\
\begin{proof}
We first prove the first part.  The assertion is simply
followed from the change of variables: $t\mapsto\tilde{t}:=\frac{t}{h}$.
Define $\tilde{\xi}(\tilde{t}):=\xi(t)=\xi(h\tilde{t})$. Then for
any $0\leq k\leq n$ we have 
\[
\tilde{\xi}^{(k)}(\tilde{t})=h^{k}\xi^{(k)}(t).
\]
Substituting this into the integral, we obtain 
\[
\int_{0}^{h}|\xi^{(n)}(t)|^{2}\,dt=h^{1-2n}\,\int_{0}^{1}|\tilde{\xi}^{(n)}(\tilde{t})|^{2}\,d\tilde{t},
\]
and the boundary conditions become 
\begin{align*}
 & (\tilde{\xi},\tilde{\xi}',\ldots,\tilde{\xi}^{(n-1)})(0)=(\xi,h\xi',\ldots,h^{n-1}\xi^{(n-1)})(0)=(x_{0},hx_{1},\ldots,h^{n-1}x_{n-1}),\\
 & (\tilde{\xi},\tilde{\xi}',\ldots,\tilde{\xi}^{(n-1)})(1)=(\xi,h\xi',\ldots,h^{n-1}\xi^{(n-1)})(h)=(y_{0},hy_{1},\ldots,h^{n-1}y_{n-1}).
\end{align*}
The assertion \eqref{eq: Ch1} then follows from these computations.

Next we prove the second statement \eqref{monotone}. The minimizing problem \eqref{eq: Ch cost} is of the form
\[
\inf_{\xi}\int_0^h L(t,\xi, \xi',\ldots,\xi^{(n)})\,dt,
\]
where $L:[0,h]\times \R^{d(n+1)}\to \R,~ L(t,x,p_1,\ldots,p_n)=|p_n|^2$. Since $L$ depends only on $p_n$ and $p_n\mapsto L(t,x,p_1,\ldots,p_n)$ is positive, continuous and convex; the existence and uniqueness of a minimizer follows from the direct method in the calculus of variations. Let $\xi_{opt}\in C^{n}([0,h],\R^{d})$
be the optimal curve in the definition of $C_{n,h}(x_{0},x_{1},\ldots,x_{n};y_{0},y_{1},\ldots,y_{n-1})$.
We define $\eta(t):=\xi_{opt}'$. Since $\eta\in C^{n-1}([0,h],\R^{d})$
and 
\begin{align*}
(\eta,\ldots,\eta^{(n-1)})(0)=(x_{1},\ldots,x_{n-1}),\quad(\eta,\ldots,\eta^{(n-1)})(h)=(y_{1},\ldots,y_{n-1}),
\end{align*}
it follows that $\eta$ is an admissible curve in the definition of
$C_{n-1,h}(x_{1},\ldots,x_{n-1};y_{1},\ldots,y_{n-1})$. It implies
that 
\begin{align*}
C_{n-1,h}(x_{1},\ldots,x_{n-1};y_{1},\ldots,y_{n-1}) & \leq \int_{0}^{h}|\eta^{(n-1)}(t)|^{2}\,dt\\
 & =\int_{0}^{h}|\xi_{opt}^{(n)}(t)|^{2}\,dt\\
 & =C_{n,h}(x_{0},x_{1},\ldots,x_{n-1};y_{0},y_{1},\ldots,y_{n-1}).
\end{align*}
This finishes the proof of the theorem. 
\end{proof}

The next lemma shows that among all the cost functions, only $\sqrt{C_{1,h}}$ is a distance in the Euclidean space $\R^{dn}$. 
\begin{lemma}
\label{lem: C=0}
 $C_{n,h}(x_0,\ldots, x_{n-1};y_0,\ldots,y_{n-1})$ is always non-negative and it equals $0$ if and only if 
\begin{equation}
\label{eq: C=0}
y_j=\sum_{i=j}^{n-1}\frac{h^{i-j}}{(i-j)!}x_i
\end{equation}
for $j=0,\ldots,n-1$. In particular, for any $z\in\R^{dn}$ we have $C_{n,h}(0,z)\geq 0$ and $C_{n,h}(0,z)=0$ iff $z=0$.
\end{lemma}
\begin{proof}
This Lemma is a direct consequence of the definition of $C_{n,h}$. Obviously $C_{n,h}\geq 0$ and it is equal to $0$ if and only if the optimal curve satisfies
\begin{equation}
\xi^{(n)}(t)=0,
\end{equation}
i.e., it is a polynomial of order $n-1$. This together with the boundary conditions imply \eqref{eq: C=0}.
\end{proof}
\begin{remark} Lemma \ref{lem: C=0} shows that $\sqrt{C_{n,h}}:\R^{dn}\times\R^{dn}\to [0,\infty)$ is not a distance when $n\geq 2$. It is not symmetric and does not satisfy the condition $\sqrt{C_{n,h}(x_0,\ldots,x_{n-1};y_0,\ldots,y_{n-1})}=0$ iff $(x_0,\ldots,x_{n-1})=(y_0,\ldots,y_{n-1})$.
\end{remark}
\subsection{Interpretation of the cost function based on  large-deviation principles and further discussions}
\label{sec: futher dis}
In this section, we provide an interpretation of the cost function based on a small-noise large-deviation principle for a special system of stochastic differential equations (SDEs) and further discussion. We consider a system of $n$ coupled oscillators, each of them moving vertically and being connected to their nearest neighbours, the last oscillator being forced by a random noise. Mathematically, the system is given by the following system of SDEs with $t\in[0,h]$,
\begin{align}
&d\xi=\xi_2\,dt\nonumber
\\&d\xi_2=\xi_3\,dt\nonumber
\\&\quad\vdots\label{eq: SDE1}
\\&d\xi_{n-1}=\xi_{n}\,dt\nonumber
\\&d\xi_n=\sqrt{\varepsilon}\, dW(t),\nonumber
\end{align}
where $W(t)$ is a $d$-dimensional Wiener process. The parameter $\varepsilon$ represents the amplitude of the noise.  This system can be formally written as
\begin{equation}
\xi_\varepsilon^{(n)}(t)=\sqrt{\varepsilon}\frac{dW}{dt}(t),
\label{eq: SDE}
\end{equation}
where the subscript indicates the dependence on $\varepsilon$. Suppose further that initial and terminal points are imposed as in \eqref{eq: boundary conditions}.  Now we consider the small-noise limit of \eqref{eq: SDE}. By the Freidlin-Wentzell theory and the contraction principle (see e.g.~\cite[Theorem 5.6.3]{DemboZeitouni98}), the process $(\xi_\varepsilon)_{\varepsilon>0}$ satisfies a large-deviation principle with a rate functional $I$ given by
\begin{equation}
\label{eq: rate functional}
I(h;\mathbf{x},\mathbf{y})=\inf\limits_{\xi}\Big\{\frac{1}{2}\int_0^h|\xi^{(n)}(t)|^2\,dt: \xi\in C^{n}([0,h],\R^d)~\text{satisfying}~\eqref{eq: boundary conditions} \Big\}.
\end{equation}
Up to a multiplicative constant, this rate functional is exactly the cost function $C_{n,h}(\mathbf{x},\mathbf{y})$ defined in \eqref{eq: Ch cost}. 

In \cite{DelarueMenozzi10} the authors considered a more general system where the right-hand sides are general functions of variables
\begin{align}
&d\xi=F_1(t,\xi,\xi_2)\,dt\nonumber
\\&d\xi_2=F_2(t,\xi,\xi_2,\xi_3)\,dt\nonumber
\\&\quad\vdots\label{eq: SDE3}
\\&d\xi_i=F_i(t,\xi,\xi_2,\ldots,\xi_{i+1})\,dt\nonumber
\\&\quad\vdots\nonumber
\\&d\xi_n=F_n(t,\xi,\ldots,\xi_n)\,dt+\sigma(t,\xi,\ldots,\xi_n)\, dW(t).\nonumber
\end{align}
Equation \eqref{eq: SDE1} is a special (linear/Gaussian) case of \eqref{eq: SDE3} with $F_i=\xi_{i+1}$, for $i=1,\ldots,n-1$, $F_n=0$ and $\sigma=\sqrt{\varepsilon}$. The reference \cite[Theorem 1.1]{DelarueMenozzi10} provides two sided Gaussian bounds for the fundamental the solution of the forward Kolmogorov equation associated to \eqref{eq: SDE3}, thus generalizes Aronson's estimate for uniformly elliptic diffusion processes \cite{Aronson} to a general hypo-elliptic setting. The rate functional $I$ above (and hence the cost function $C_{n,h}$ of the present paper) plays a key role in \cite{DelarueMenozzi10} because of two reasons: 
\begin{enumerate}[(1)]
\item It is related to the fundamental solution of the forward Kolmogorov equation associated to the Gaussian system which is proved using Fleming's logarithmic transform and control theory \cite[Proposition 3.1]{DelarueMenozzi10}.
\item \cite[Theorem 1.1]{DelarueMenozzi10} is first proved for the Gaussian case using (1) then extended to the nonlinear case by linearizing.
\end{enumerate}
In addition, in the case $n=2$ the cost function $C_{2,h}$ (particularly its explicit formulation) has also been used to construct variational approximation schemes for the Kramer equation \cite{Hua00,DPZ13a}, the system of isentropic Euler equations \cite{GW09, Westdickenberg10} and the compressible Euler equations \cite{CSW2014TMP}.

Based on the representation obtained in Theorem 1.2 of the present paper, in a companion paper \cite{DuongTran2017} we provide an elementary proof for \cite[Proposition 3.1]{DelarueMenozzi10} and extend \cite{Hua00,DPZ13a} to the forward Kolmogorov equation associated to \eqref{eq: SDE1} of which the Kramers equation is a special case.
\section{Analytical formula of the cost functions}
\label{sec: formula}
In this section, we prove Theorem \ref{theo: explicit formula} on the explicit formula for the cost function. Throughout the rest of the paper, all indices are numbered starting with zero.
\\ \ \\
\begin{proof}\textit{\textbf{of Theorem \ref{theo: explicit formula}}}.
We first recall the definition of a Wronskian matrix that will be used at several places later on. The Wronskian matrix $W(f_{1},\ldots,f_{n})$  associated to $n$ functions $f_{1},\ldots,f_{n}$ of a single real variable (say time) in the class $C^{n}$
is defined by\footnote{Note that other authors sometimes denote by $W(f_{1},\ldots,f_{n})$ the determinant of the Wronskian matrix defined here.} 
\begin{equation*}
W(f_{1},\ldots,f_{n})=\begin{pmatrix}f_{1} & \ldots & f_{n}\\
f_{1}' & \ldots & f_{n}'\\
\vdots & \vdots & \vdots\\
f_{1}^{(n-1)} & \ldots & f_{n}^{(n-1)}
\end{pmatrix},
\end{equation*}  
so that the $(i,j)^{\mathrm{th}}$-entry of this matrix is $f_{j}^{(i)}$, which is the $i^{\mathrm{th}}$-order derivative of $f_{j}$.

The optimal curve $\xi$ in the definition of the cost function $C_{n,h}$ satisfies the Euler-Lagrange equation
\begin{equation}
\xi^{(2n)}(t)=0.
\end{equation}
Therefore, it is a polynomial of order $2n-1$
\begin{equation}
\xi(t)=\sum_{i=0}^{2n-1}a_i t^i.
\end{equation}
The coefficients $\{a_i\}_{i=0}^{2n-1}$ will be determined from the boundary conditions. The $k^{\mathrm{th}}$-order  derivative of $\xi$ can be easily computed as
\begin{align*}\xi^{(k)}(t) & =\sum_{i=k}^{2n-1}k!\begin{pmatrix}i\\
k
\end{pmatrix}a_{i}t^{i-k}\\
= & \left(\begin{array}{cccccccc}
0 & \cdots & \underbrace{0}_{(k-1)^{\mathrm{th}}\text{ entry}} & k! & \cdots & \underbrace{k!\begin{pmatrix}i\\
k
\end{pmatrix}t^{i-k}}_{i^{\mathrm{th}}\text{ entry, }k\le i\le n-1} & \cdots & k!\begin{pmatrix}n-1\\
k
\end{pmatrix}t^{n-1-k}\end{array}\right)\left(\begin{array}{c}
a_{0}\\
a_{1}\\
\vdots\\
a_{n-1}
\end{array}\right)+\\
 & +\left(\begin{array}{ccccc}
k!\begin{pmatrix}n\\
k
\end{pmatrix}t^{n-k} & \cdots & \underbrace{k!\begin{pmatrix}i\\
k
\end{pmatrix}t^{i-k}}_{(i-n+1)^{\mathrm{th}}\text{ entry, }n\le i\le2n-1} & \cdots & k!\begin{pmatrix}2n-1\\
k
\end{pmatrix}t^{2n-1-k}\end{array}\right)\left(\begin{array}{c}
a_{n}\\
a_{n+1}\\
\vdots\\
a_{2n-1}
\end{array}\right).
\end{align*}
Writing these equations for $k=0,\ldots, n-1$ in matrix form, we obtain
\begin{eqnarray}
\left(\begin{array}{c}
\xi(t)\\
\xi'(t)\\
\vdots\\
\xi^{(n-1)}(t)
\end{array}\right) & = & V_{n}\left(\begin{array}{c}
a_{0}\\
a_{1}\\
\vdots\\
a_{n-1}
\end{array}\right)+A_{n}\left(\begin{array}{c}
a_{n}\\
a_{n+1}\\
\vdots\\
a_{2n-1}
\end{array}\right),\label{eq1_1}
\end{eqnarray}
where the matrices $V_n$ and $A_n$ depend on $t$ and are given by 
\begin{eqnarray*}
V_{n}(t) & = & \left[\begin{array}{ccccccc}
1 & \begin{pmatrix}2\\
0
\end{pmatrix}t & \begin{pmatrix}3\\
0
\end{pmatrix}t^{2} & \cdots & \begin{pmatrix}k\\
0
\end{pmatrix}t^{k-1} & \cdots & \begin{pmatrix}n-1\\
0
\end{pmatrix}t^{n-2}\\
0 & 1! & 1!\begin{pmatrix}3\\
1
\end{pmatrix}t^{1} & \cdots & 1!\begin{pmatrix}k\\
1
\end{pmatrix}t^{k-1} & \cdots & 1!\begin{pmatrix}n-1\\
1
\end{pmatrix}t^{n-3}\\
0 & 0 & 2! & \cdots & \vdots & \cdots & 2!\begin{pmatrix}n-1\\
2
\end{pmatrix}t^{n-4}\\
\vdots & \vdots & \cdots & \ddots & \vdots & \cdots & \vdots\\
0 & 0 & \cdots & \cdots & k! & \cdots & k!\begin{pmatrix}n-1\\
k
\end{pmatrix}t^{n-1-k}\\
\vdots & \vdots & \vdots & \vdots & \vdots & \ddots & \vdots\\
0 & 0 & \cdots & \cdots & 0 & 0 & (n-1)!
\end{array}\right]
\end{eqnarray*}
and

\begin{eqnarray}
A_{n}(t)=\left[\begin{array}{cccc}
t^{n} & t^{n+1} & \cdots & t^{2n-1}\\
\begin{pmatrix}n\\
1
\end{pmatrix}t^{n-1} & \begin{pmatrix}n+1\\
1
\end{pmatrix}t^{n} & \cdots & \begin{pmatrix}2n-1\\
1
\end{pmatrix}t^{2n-2}\\
\vdots & \vdots & \vdots & \vdots\\
k!\begin{pmatrix}n\\
k
\end{pmatrix}t^{n-k} & k!\begin{pmatrix}n+1\\
k
\end{pmatrix}t^{n+1-k} & \cdots & k!\begin{pmatrix}2n-1\\
k
\end{pmatrix}t^{2n-1-k}\\
\vdots & \vdots & \vdots & \vdots\\
(n-1)!\begin{pmatrix}n\\
n-1
\end{pmatrix}t & (n-1)!\begin{pmatrix}n+1\\
n-1
\end{pmatrix}t^{2} & \cdots & (n-1)!\begin{pmatrix}2n-1\\
n-1
\end{pmatrix}t^{n}
\end{array}\right].\label{A}
\end{eqnarray}

Note that $V_n(t)$ and $A_n(t)$ can be written in compact forms using the notation of the Wronskians
\begin{equation*}
V_n(t)=W(1,t,\ldots,t^{n-1}), \quad \text{and}\quad A_n(t)=W(t^n,\ldots,t^{2n-1}).
\end{equation*}

In particular, when $t=0$, $V_n(0)=\mathrm{diag}(1,1!,2!,\ldots, (n-1)!)$ is the diagonal matrix, and $A_n(0)=0$. It follows that
\begin{eqnarray*}
\left(\begin{array}{c}
a_{0}\\
a_{1}\\
\vdots\\
a_{n-1}
\end{array}\right)=V_{n}^{-1}(0)\left(\begin{array}{c}
\xi(0)\\
\xi'(0)\\
\vdots\\
\xi^{(n-1)}(0)
\end{array}\right)=\left(\begin{array}{c}
\xi(0)\\
\frac{1}{1!}\xi'(0)\\
\vdots\\
\frac{1}{(n-1)!}\xi^{(n-1)}(0)
\end{array}\right).
\end{eqnarray*}
Similarly, when $t=h$, we obtain
\begin{eqnarray*}
\left(\begin{array}{c}
\xi(h)\\
\xi'(h)\\
\vdots\\
\xi^{(n-1)}(h)
\end{array}\right) & = & V_{n}(h)\left(\begin{array}{c}
a_{0}\\
a_{1}\\
\vdots\\
a_{n-1}
\end{array}\right)+A_{n}(h)\left(\begin{array}{c}
a_{n}\\
a_{n+1}\\
\vdots\\
a_{2n-1}
\end{array}\right).
\end{eqnarray*}
Therefore, we have the following equation to define $a_{n},...,a_{2n-1}$,
\[
A_{n}(h)\left(\begin{array}{c}
a_{n}\\
a_{n+1}\\
\vdots\\
a_{2n-1}
\end{array}\right)=\mathbf{b}_n,
\]
where the vector $\mathbf{b}_n$ on the right-hand side is given by
\begin{eqnarray*}
\mathbf{b}_n(h) & = & \left(\begin{array}{c}
\xi(h)\\
\xi'(h)\\
\vdots\\
\xi^{(n-1)}(h)
\end{array}\right)-V_{n}(h)\left(\begin{array}{c}
\xi(0)\\
\frac{1}{1!}\xi'(0)\\
\vdots\\
\frac{1}{(n-1)!}\xi^{(n-1)}(0)
\end{array}\right).
\end{eqnarray*}
The $i$-component of $\mathbf{b}_n(h)$ can be computed explicitly using the definition of $V_n(h)$ as follows
\begin{eqnarray}
\mathbf{b}_n(h)[i] & = & \xi^{(i)}(h)-\sum_{j=i}^{n-1}i!\begin{pmatrix}j\\
i
\end{pmatrix}h^{j-i}\frac{1}{j!}\xi^{(j)}(0)\nonumber\\
 & = & \xi^{(i)}(h)-\sum_{j=i}^{n-1}i!\frac{j!}{(j-i)!i!}h^{j-i}\frac{1}{j!}\xi^{(j)}(0)\nonumber\\
 & = & \xi^{(i)}(h)-\sum_{j=i}^{n-1}\frac{1}{(j-i)!}h^{j-i}\xi^{(j)}(0)\nonumber\\
  & = & y_i-\sum_{j=i}^{n-1}\frac{1}{(j-i)!}h^{j-i}x_j.\label{b}
\end{eqnarray}
To proceed, we use the following lemma, whose proof is given below.
\begin{lemma}
\label{lem: detA} The matrix $A_{n}(h)$ is invertible. 
\end{lemma}
Therefore, we can compute the coefficients $a_n,\ldots,a_{2n-1}$ from the matrix $A_n(h)$ and the vector $\mathbf{b}_n$.
\begin{equation*}
\left(\begin{array}{c}
a_{n}\\
a_{n+1}\\
\vdots\\
a_{2n-1}
\end{array}\right)=A_n(h)^{-1}\mathbf{b}_n(h).
\end{equation*}

On the other hand, by integrating by parts successively, we obtain
\begin{eqnarray}
\int_{0}^{h}|\xi^{(n)}(t)|^{2}\, dt & = & \int_{0}^{h}\xi^{(n)}(t){}^{2}\, dt\nonumber\\
 & = & \int_{0}^{h}\xi^{(n)}(t)\, d\xi^{(n-1)}(t)\nonumber\\
 & = & \xi^{(n)}(t)\xi^{(n-1)}(t)\Big\vert_{0}^{h}-\int_{0}^{h}\xi^{(n-1)}(t)\, d\xi^{(n)}(t)\nonumber\\
 & = & \xi^{(n)}(t)\xi^{(n-1)}(t)\Big\vert_{0}^{h}-\int_{0}^{h}\xi^{(n-1)}(t)\xi^{(n+1)}(t)\, dt\nonumber\\
 & = & \xi^{(n)}(t)\xi^{(n-1)}(t)\Big\vert_{0}^{h}-\int_{0}^{h}\xi^{(n+1)}(t)\, d\xi^{(n-2)}(t)\nonumber\\
 & = & \left(\xi^{(n)}(t)\xi^{(n-1)}(t)-\xi^{(n+1)}(t)\xi^{(n-2)}(t)\right)\Big\vert_{0}^{h}+\int_{0}^{h}\xi^{(n-2)}(t)\, d\xi^{(n+1)}(t)\nonumber\\
 & = & \left(\xi^{(n)}(t)\xi^{(n-1)}(t)-\xi^{(n+1)}(t)\xi^{(n-2)}(t)\right)\Big\vert_{0}^{h}+\int_{0}^{h}\xi^{(n-2)}(t)\xi^{(n+2)}(t)\, dt\nonumber\\
 & = & \cdots\nonumber\\
 & = & \sum_{i=0}^{n-1}(-1)^{i}\xi^{(n+i)}(t)\xi^{(n-1-i)}(t)\Big\vert_{0}^{h}+\int_{0}^{h}\xi(t)\xi^{(2n)}(t)\, dt\nonumber\\
 & = & \sum_{i=0}^{n-1}(-1)^{i}\xi^{(n+i)}(t)\xi^{(n-1-i)}(t)\Big\vert_{0}^{h},\label{eq: Ch via xi}
\end{eqnarray}
where we have used the fact that $\xi^{(2n)}(t)=0$ to obtain the last equality. Next we will compute the last expression using the relation \eqref{eq1_1}. 
It follows from \eqref{eq1_1} that 
\begin{eqnarray*}
\left(\begin{array}{c}
\xi^{(n)}(t)\\
\xi^{(n+1)}(t)\\
\vdots\\
\xi^{(n+k)}(t)\\
\vdots\\
\xi^{(2n-1)}(t)
\end{array}\right) & = & A_{n}^{(n)}(t)\left(\begin{array}{c}
a_{n}\\
a_{n+1}\\
\vdots\\
a_{k}\\
\vdots\\
a_{2n-1}
\end{array}\right)
\end{eqnarray*}
where $A_{n}^{(n)}(t)$ is the matrix obtained from $A_n(t)$ by taking $n^{\mathrm{th}}$-order derivative of each entry of $A_n(t)$. The $(k,i)^{\mathrm{th}}$-element, $i=1,\ldots, n, k=1,\ldots, n, i\geq k$, of $A_n^{(n)}(t)$ is given by
\begin{align*}
\left((k-1)!\begin{pmatrix}n+i-1\\
k-1
\end{pmatrix}t^{n+i-k}\right)^{(n)}= & (k-1)!\begin{pmatrix}n+i-1\\
k-1
\end{pmatrix}n!\begin{pmatrix}i-k+n\\
i-k
\end{pmatrix}t^{-k+i}\\
= & \frac{(n+i-1)!}{(i-k)!}t^{i-k}.
\end{align*}
Other elements of $A_n^{(n)}(t)$ are equal to 0. 
Therefore, we have
\begin{align*}
\left(\begin{array}{c}
\xi^{(n)}(t)\xi^{(n-1)}(t)\\
\xi^{(n+1)}(t)\xi^{(n-2)}(t)\\
\cdots\\
\xi^{(n+k)}(t)\xi^{(n-k-1)}(t)\\
\cdots\\
\xi^{(2n-1)}(t)\xi(t)
\end{array}\right)
=D_n(t)\left(\begin{array}{c}
a_{n}\\
a_{n+1}\\
\vdots\\
a_{k}\\
\vdots\\
a_{2n-1}
\end{array}\right),
\end{align*}
where the $(k,i)^{\mathrm{th}}$-element, $i=1,\ldots, n, k=1,\ldots, n, i\geq k$, of $D_n(t)$ is given by $\frac{(n+i-1)!}{(i-k)!}t^{i-k}\xi^{(n-k)}(t)$. Other elements of $D_n(t)$ are equal to 0.
It follows that
\begin{equation*}
\left(\begin{array}{c}
\xi^{(n)}(h)\xi^{(n-1)}(h)\\
\xi^{(n+1)}(h)\xi^{(n-2)}(h)\\
\vdots\\
\xi^{(n+k)}(h)\xi^{(n-k-1)}(h)\\
\vdots\\
\xi^{(2n-1)}(h)\xi(h)
\end{array}\right)-\left(\begin{array}{c}
\xi^{(n)}(0)\xi^{(n-1)}(0)\\
\xi^{(n+1)}(0)\xi^{(n-2)}(0)\\
\vdots\\
\xi^{(n+k)}(0)\xi^{(n-k-1)}(0)\\
\vdots\\
\xi^{(2n-1)}(0)\xi(0)
\end{array}\right)= E_{n}(h)[A_{n}(h)]^{-1}\mathbf{b}_n(h),
\end{equation*}
where the matrix $E_{n}(h)=D_n(h)-D_n(0)$ is given by
\[
E_{n}(h)[i_{1},i_{2}]=\begin{cases}
\frac{(n+i_{2})!}{(i_{2}-i_{1})!}\xi^{(n-i_{1}-1)}(h)h^{-i_{1}+i_{2}}, & i_{2}>i_{1},\\
\frac{(n+i_{2})!}{(i_{2}-i_{1})!}[\xi^{(n-i_{1}-1)}(h)-\xi^{(n-i_{1}-1)}(0)], & i_{2}=i_{1},\\
0, & i_{2}<i_{1},
\end{cases}
\]
for all $i_{1},i_{2}=0,\ldots,n$.
Therefore, substituting this back to \eqref{eq: Ch via xi}, we obtain
\begin{equation}
C_{h}=\int_{0}^{h}(\xi^{(n)}(t))^{2}dt=\left(\begin{array}{ccccc}
1 & -1 & 1 & -1 & \cdots\end{array}\right)E_{n}(h)[A_{n}(h)]^{-1}\mathbf{b}_n(h).\label{eq1-1}
\end{equation}
The right-hand side of \eqref{eq1-1} can be transformed further using the following lemma, whose proof is presented below.
\begin{lemma} 
\label{lem: Bn}
It holds that
\begin{eqnarray}
\left[\begin{array}{ccccc}
1 & -1 & 1 & -1 & \cdots\end{array}\right]E_{n}(h) & = & [\mathbf{b}_n(h)]^{T}B_{n}(h)\label{eq2},
\end{eqnarray}
where the matrix $B_{n}(h)$ is defined as follows
\begin{equation}
B_{n}(h)[i_{1},i_{2}]=\begin{cases}
(-1)^{n-i_{1}-1}\frac{(n+i_{2})!}{(i_{1}+i_{2}-n+1)!}h^{i_{2}+i_{1}-n+1}, & i_{2}+i_{1}\ge n-1,\\
0 & i_{2}+i_{1}<n-1,
\end{cases}\label{B}
\end{equation}
for all $i_{1},i_{2}=0,\ldots,n-1$.
\end{lemma}
Substituting \eqref{eq2} into \eqref{eq1-1}, we obtain 
\begin{equation*}
C_{h}=\int_{0}^{h}(\xi^{(n)}(t))^{2}dt=[\mathbf{b}_n(h)]^{T}B_{n}(h)[A_{n}(h)]^{-1}\mathbf{b}_n(h),
\end{equation*}
with $\mathbf{b}_n(h),  A_n(h)$ and $B_n(h)$ defined in \eqref{b}, \eqref{A} and \eqref{B} respectively. This establishes the statement of Theorem \ref{theo: explicit formula}. 
\end{proof}

For completion, we now prove Lemma \ref{lem: detA} and Lemma \ref{lem: Bn}.

\begin{proof}\textit{\textbf{of Lemma \ref{lem: detA}}}. We recall that $A_{n}(h)$ can be written in terms of the Wronskian
\begin{equation*}
A_{n}(h)=W(h^n,\ldots,h^{2n-1}).
\end{equation*}
Therefore 
\begin{equation*}
\det A_n(h)=\det W(f_{1},\ldots,f_{n}).
\end{equation*}
According to \cite[Lemma 1]{BostanDumas2010}, we have 
\begin{equation*}
\det A_n(h)=V(n,\ldots,2n-1)h^{\footnotesize{\sum\limits _{i=n}^{2n-1}i-\begin{pmatrix}n\\
2
\end{pmatrix}}}=V(n,\ldots,2n-1)h^{\frac{n(3n-1)}{2}},
\end{equation*}
where $V(n,\ldots,2n-1)$ is the Vandermonde determinant 
\begin{equation*}
V(n,\ldots,2n-1)=\left|\begin{array}{ccc}
1 & \ldots & 1\\
n & \ldots & 2n-1\\
\vdots & \vdots & \vdots\\
n^{n-1} & \ldots & (2n-1)^{n-1}
\end{array}\right|=\prod\limits _{1\leq i<j\leq n}(j-i).
\end{equation*}
Hence, we obtain
\begin{equation}
\det A_n(h)=h^{\frac{n(3n-1)}{2}}\prod\limits _{1\leq i<j\leq n}(j-i),
\end{equation}
which is non-zero.
This completes the proof of the lemma.
\end{proof} 

\begin{proof}\textit{\textbf{of Lemma \ref{lem: Bn}}}. The equality \eqref{eq2} is interesting on its own. Below we will prove it using purely combinatorial techniques.

The $k^{\mathrm{th}}$ element on the left hand side of \eqref{eq2} is
\begin{align*}
\sum_{i=0}^{n-1}(-1)^{i}E_{n}(h)[i,k]=(n+k)!(-1)^{k}(\xi^{(n-k-1)}(h)-\xi^{(n-k-1)}(0))+\sum_{i=0}^{k-1}(-1)^{i}\frac{(n+k)!}{(k-i)!}\xi^{(n-i-1)}(h)h^{-i+k}.
\end{align*}
The $(i_{1},i_{2})$ element of $B_{n}(h)$ is 
\[
B_{n}(h)[i_{1},i_{2}]=\begin{cases}
(-1)^{n-i_{1}-1}\frac{(n+i_{2})!}{(i_{1}+i_{2}-n+1)!}h^{i_{2}+i_{1}-n+1}, & i_{2}+i_{1}\ge n-1,\\
0 & i_{2}+i_{1}<n-1.
\end{cases}
\]

So that the $k^\mathrm{th}$ element on the right hand side of \eqref{eq2} is 
\begin{eqnarray*}
\sum_{i=0}^{n-1}\mathbf{b}_n(i)B_{n}[i,k] & = & \sum_{i=0,i+k\ge n-1}^{n-1}\mathbf{b}_n(i)(-1)^{n-i-1}\frac{(n+k)!}{(k+i-n+1)!}h^{k+i-n+1}\\
 & = & \sum_{i=n-1-k}^{n-1}\mathbf{b}_n(i)(-1)^{n-i-1}\frac{(n+k)!}{(k+i-n+1)!}h^{k+i-n+1}\\
 & = & \sum_{i=n-1-k}^{n-1}[\xi^{(i)}(h)-\sum_{j=i}^{n-1}\frac{1}{(j-i)!}h^{j-i}\xi^{(j)}(0)](-1)^{n-i-1}\frac{(n+k)!}{(k+i-n+1)!}h^{k+i-n+1}.
\end{eqnarray*}

To establish \eqref{eq2}, we need to show that
\begin{align}
&\sum_{i=n-1-k}^{n-1}[\xi^{(i)}(h)-\sum_{j=i}^{n-1}\frac{1}{(j-i)!}h^{j-i}\xi^{(j)}(0)](-1)^{n-i-1}\frac{(n+k)!}{(k+i-n+1)!}h^{k+i-n+1}\nonumber
\\&=(n+k)!(-1)^{k}(\xi^{(n-k-1)}(h)-\xi^{(n-k-1)}(0))+\sum_{i=0}^{k-1}(-1)^{i}\frac{(n+k)!}{(k-i)!}\xi^{(n-i-1)}(h)h^{-i+k}.\label{eq: equality}
\end{align}
We will show this by transforming the left-hand side . First we write it as follows
\begin{align*}
&\sum_{i=n-1-k}^{n-1}[\xi^{(i)}(h)-\sum_{j=i}^{n-1}\frac{1}{(j-i)!}h^{j-i}\xi^{(j)}(0)](-1)^{n-i-1}\frac{(n+k)!}{(k+i-n+1)!}h^{k+i-n+1}
\\&=\sum_{i=n-1-k}^{n-1}\xi^{(i)}(h)(-1)^{n-i-1}\frac{1}{(k+i-n+1)!}h^{k+i-n+1}
\\&\qquad -\sum_{i=n-1-k}^{n-1}\sum_{j=i}^{n-1}\frac{1}{(j-i)!}h^{j-i}\xi^{(j)}(0)(-1)^{n-i-1}\frac{1}{(k+i-n+1)!}h^{k+i-n+1} 
\\&=(I)-(II).
\end{align*}
Now we transform further $(I)$ and $(II)$. We have
\begin{align}
(I)&:=\sum_{i=n-1-k}^{n-1}\xi^{(i)}(h)(-1)^{n-i-1}\frac{1}{(k+i-n+1)!}h^{k+i-n+1}\nonumber\\
&=\sum_{i=n-1-k}^{n-1}\xi^{(i)}(h)(-1)^{n-i-1}\frac{1}{(k+i-n+1)!}h^{k+i-n+1}\text{ (now we change variable: }i=-j+n-1)\nonumber\\
&=\sum_{i=k}^{0}\xi^{(n-j-1)}(h)(-1)^{j}\frac{1}{(k-j)!}h^{k-j}\nonumber\\
&=\xi^{(n-k-1)}(h)(-1)^{k}+\sum_{i=0}^{k-1}\xi^{(n-j-1)}(h)(-1)^{j}\frac{1}{(k-j)!}h^{k-j}.\label{eq:equality2}
\end{align}
The term $(II)$ can be transformed as follows.
\begin{align}
(II)&:=-\sum_{i=n-1-k}^{n-1}\sum_{j=i}^{n-1}\frac{1}{(j-i)!}h^{j-i}\xi^{(j)}(0)(-1)^{n-i-1}\frac{1}{(k+i-n+1)!}h^{k+i-n+1} \nonumber
\\&=\sum_{i=n-1-k}^{n-1}\sum_{j=i}^{n-1}\frac{1}{(j-i)!}h^{j-i}\xi^{(j)}(0)(-1)^{n-i}\frac{1}{(k+i-n+1)!}h^{k+i-n+1}\nonumber
\\&=\sum_{j=n-1-k}^{n-1}\xi^{(j)}(0)h^{j+k-n+1}\sum_{i=n-1-k}^{j}\frac{1}{(j-i)!}(-1)^{n-i}\frac{1}{(k+i-n+1)!}(\text{ for }j=n-1-k..n-1)\nonumber
\\&=-\xi^{(n-k-1)}(0)(-1)^k+\sum_{j=n-k}^{n-1}\xi^{(j)}(0)h^{j+k-n+1}\sum_{i=n-1-k}^{j}\frac{1}{(j-i)!}(-1)^{n-i}\frac{1}{(k+i-n+1)!}.
\label{eq: equality3}
\end{align}

We now show that the second term of \eqref{eq: equality3} vanishes by showing that
\begin{equation*}
\sum_{i=n-1-k}^{j}\frac{1}{(j-i)!}(-1)^{n-i}\frac{1}{(k+i-n+1)!}=0
\end{equation*}
for each $j=n-k,\ldots, n-1$. In fact, we have
\begin{align*}
&\sum_{i=n-1-k}^{j}(-1)^{n-i}\frac{1}{(j-i)!(k+i-n+1)!}\\
&\qquad=\frac{1}{(j+k-n+1)!}\sum_{i=n-1-k}^{j}(-1)^{n-i}\frac{(j+k-n+1)!}{(j-i)!(k+i-n+1)!}\\
&\qquad=\frac{(-1)^{n-j}}{(j+k-n+1)!}\sum_{i=n-1-k}^{j}(-1)^{j-i}\begin{pmatrix}j+k-n+1\\
j-i
\end{pmatrix}\\
&\qquad=\frac{(-1)^{n-j}}{(j+k-n+1)!}\sum_{i=j-(n-1-k)}^{0}(-1)^{l}\begin{pmatrix}j+k-n+1\\
l
\end{pmatrix}\text{(now we change variable: }l=j-i,i=j-l)\\
&\qquad=\frac{(-1)^{n-j}}{(j+k-n+1)!}\sum_{l=0}^{j-n+1+k}(-1)^{l}\begin{pmatrix}j+k-n+1\\
l
\end{pmatrix}\\
&\qquad=\frac{(-1)^{n-j}}{(j+k-n+1)!}(-1+1)^{j-n+k+1}\\
\\&\qquad=0.
\end{align*}
As a consequence, $(II)=-\xi^{(n-k-1)}(0)(-1)^{k}$. Therefore, the left-hand side of \eqref{eq: equality} is equal to 
\begin{equation*}
(\xi^{(n-k-1)}(h)-\xi^{(n-k-1)}(0))(-1)^{k}+\sum_{i=0}^{k-1}\xi^{(n-j-1)}(h)(-1)^{j}\frac{1}{(k-j)!}h^{k-j},
\end{equation*}
which is exactly its right-hand side. This finishes the proof of Lemma \ref{lem: Bn}.
\end{proof}
\begin{remark}
We recall that the cost function $C_{n,h}$ is given by
\begin{equation}
\label{eq: final formula of C}
C_{n,h}=\mathbf{b}_n(h)^TB_n(h)A_n^{-1}(h)\mathbf{b}_n(h).
\end{equation}
We can show that the matrix $B_n(h)A_n^{-1}(h)$ is positive definite. Indeed, let $z\in\R^{dn}$ be arbitrary. Then by definition of $C_{n,h}$ and from \eqref{eq: final formula of C}, we have
\begin{equation*}
z^TB_n(h)A_n^{-1}(h)z=C_{n,h}(0,z).
\end{equation*}
Then the positive definiteness of $B_n(h)A_n^{-1}(h)$ follows from Lemma \ref{lem: C=0}. Note further that we have the following property, for any matrix $M$ then
\begin{equation}
z^TMz=(z^TMz)^T=z^TM^Tz=z^T\frac{M+M^T}{2}z.
\end{equation}
Therefore we can write
\begin{equation*}
C_{n,h}=\mathbf{b}_n(h)^TH_n(h)\mathbf{b}_n(h),
\end{equation*}
where $H_n(h)=\frac{1}{2}[B_n(h)A_n^{-1}(h)+(B_n(h)A_n^{-1}(h))^T]$ is a symmetric positive matrix. 
\end{remark}
\begin{remark}[Alternative computations] We can compute the cost function $C_{n,h}$ more directly as follows.
\begin{align*}
\int_0^h|\xi^{(n)}(t)|^2\,dt&=\int_0^h\left(\sum_{i=n}^{2n-1}n!\begin{pmatrix}
i\\n
\end{pmatrix}a_it^{i-n}\right)^2\,dt
\\&=(n!)^2\sum_{n\leq i,j\leq 2n-1}\begin{pmatrix}
i\\n
\end{pmatrix}\begin{pmatrix}
j\\n
\end{pmatrix}a_ia_j\int_0^h t^{i+j-2n}\,dt
\\&=(n!)^2\sum_{n\leq i,j\leq 2n-1}\begin{pmatrix}
i\\n
\end{pmatrix}\begin{pmatrix}
j\\n
\end{pmatrix}a_ia_j \frac{h^{i+j+1-2n}}{i+j+1-2n}.
\end{align*}
Define the matrix $K_n(h)=(K_{ij})_{i,j=n}^{2n-1}$ with entries
\begin{equation*}
K_{ij}=(n!)^2\begin{pmatrix}
i\\n
\end{pmatrix}\begin{pmatrix}
j\\n
\end{pmatrix}\frac{h^{i+j+1-2n}}{i+j+1-2n}.
\end{equation*}
Clearly $K_n(h)$ is symmetric.
Then $C_{n,h}$ can be written as
\begin{equation*}
C_{n,h}=a^TK_n(h)a=\mathbf{b}_n^T(A_{n}(h)^{-1})^T K_n(h) A_n(h)^{-1}\mathbf{b}_n(h),
\end{equation*}
where $a=(a_n,\ldots,a_{2n-1})^T$ and $\mathbf{b}_n(h)$  and $A_n(h)$ are defined in \eqref{B} and \eqref{A} respectively.

The advantage of the formula derived in the previous section is that it involves $A_n^{-1}$ only one time and the matrix $B_n$ is triangular.
\end{remark}
\begin{remark}[Alternative representation using the scaling property] Using the scaling property in Theorem \ref{theo: qual theo 2}, the cost function can be written as follows
\begin{equation*}
C_{n,h}=\,h^{1-2n}\,[\tilde{\mathbf{b}}_n(h)]^TB_n(1)[A_n(1)]^{-1}\tilde{\mathbf{b}}_n(h),
\end{equation*}
where 
\begin{equation*}
\tilde{\mathbf{b}}_n(h)=\begin{pmatrix}
\tilde{\xi}(1)\\
h\tilde{\xi}'(1)\\
\vdots\\
h^k\tilde{\xi}^{(k)}(1)\\
\vdots\\
h^{n-1}\tilde{\xi}^{(n-1)}(1)
\end{pmatrix}-V_n(1)\begin{pmatrix}
\tilde{\xi}(0)\\
h\tilde{\xi}'(0)\\
\vdots\\
h^k\tilde{\xi}^{(k)}(0)\\
\vdots\\
h^{n-1}\tilde{\xi}^{(n-1)}(0)
\end{pmatrix}.
\end{equation*}
By analogous computations to \eqref{b}, the $i$-component of this vector is
\begin{equation*}
\tilde{\mathbf{b}}_n(h)[i]=y_ih^i-\sum_{j=i}^{n-1}\frac{1}{(j-i)!}h^{j}x_j=h^i\mathbf{b}_n(h)[i].
\end{equation*}
\end{remark}
\section{ LU decomposition of $A_{n}$ and $A_n^{-1}$}
\label{sec: LU} The analytical formulas for the cost functions obtained
in the previous section involve the inverse of the matrix $A_{n}$,
which is a matrix of order $n$. In this section, we provide an explicit formula for the $LU$ decomposition of $A_n$ and for $A_n^{-1}$. For simplicity of notation, we leave out the dependence on $h$ of $A_n$ (and hence $L, U$). We recall that $A=W(h^{n},\ldots,h^{2n-1})$ is the Wronskian matrix associated to the polynomials $\{h^n,\ldots,h^{2n-1}\}$, therefore the analysis of this section is of independent interest since the Wronskian matrix plays an important role in linear algebra and differential equations.

The main result of this section is the following theorem, which is summarised as Theorem \ref{theo: LU of A} in the introduction. 
\begin{theorem}
\label{theo: LU of A 2}
\begin{enumerate}
\item  $A_{n}=LU$ where $U$ and $L$ are defined as follows
\begin{equation}
\label{U}
U[i,j]= \begin{cases}
\frac{(j-1)!}{(j-i)!}h^{j+n-i} & \text{ if }j\ge i,\\
0 & \text{otherwise},
\end{cases} 
\end{equation}
and 
\begin{equation}
\label{L}
L[k,j]=\begin{cases}
h^{j-k}\begin{pmatrix}k-1\\
j-1
\end{pmatrix}\frac{n!}{(n-k+j)!} & \text{ if }j\le k,\\
0 & \text{otherwise}.
\end{cases}
\end{equation}
\item The inverse of $A_n$ is given by the product of the following two matrices:
\begin{equation}
U^{-1}[i,j]=\begin{cases}
\frac{(-1)^{i+j}}{((i-1)!(j-i)!)h^{-j+i+n}} & \text{ if }j\ge i,\\
0 & \text{otherwise},
\end{cases}\label{heq3}
\end{equation}
and
\begin{align}
L^{-1}[j,i]&=\begin{cases}
(-1)^{j-i}h^{i-j}\frac{(j-1)!}{(j-i)!(i-1)!}\frac{(n+j-i-1)!}{(n-1)!} & \text{ if }j\ge i,\\
0 & \text{otherwise},
\end{cases}\label{eq21}
\\&=\begin{cases}
(-1)^{j-i}h^{i-j}\frac{(j-1)!}{(i-1)!}\begin{pmatrix}
n+j-i-1\\
j-i
\end{pmatrix}\nonumber& \text{ if }j\ge i\\
0 & \text{otherwise}.
\end{cases}
\end{align}
\end{enumerate}
\end{theorem}
\begin{proof}
We first prove the first statement about the $LU$ decomposition of the matrix $A_n$
We will show that $U^{-1}$ can be computed as follows:
\begin{equation*}
U^{-1}[i,j]=\begin{cases}
\frac{1}{((i-1)!(j-i)!)h^{-j+i+n}} & \text{ if }j\ge i,\\
0 & \text{otherwise}.
\end{cases}
\end{equation*}
In fact, we have 
\begin{eqnarray*}
\sum_{j}U[i,j]U^{-1}(j,k) & = & \sum_{j\ge i,k\ge j}U[i,j]U^{-1}[j,k]\\
 & = & \sum_{j=i}^{k}U[i,j]U^{-1}[j,k]\\
 & = & \sum_{j=i}^{k}\frac{(j-1)!}{(j-i)!}h^{j+n-i}\frac{(-1)^{j+k}}{((j-1)!(k-j)!)h^{n-k+j}}\\
 & = & h^{k-i}(-1)^{k}\sum_{j=i}^{k}\frac{1}{(j-i)!}\frac{(-1)^{j}}{(k-j)!}.
\end{eqnarray*}
If $k=i$, then $\sum\limits_{j}U[i,j]U^{-1}[j,k]=(-1)^{k}\frac{1}{(k-i)!}\frac{(-1)^{k}}{(k-k)!}=1$.
\\ \ \\
If $k\ne i$, then we have
\begin{eqnarray*}
\sum_{j=i}^{k}\frac{1}{(j-i)!}\frac{(-1)^{j}}{(k-j)!} & = & \frac{1}{(k-i)!}\sum_{j=i}^{k}\frac{(k-i)!}{(j-i)!(k-j)!}(-1)^{j}\\
 & = & \frac{1}{(k-i)!}(-1)^{i}\sum_{l=0}^{k-i}\frac{(k-i)!}{l!(k-i-l)!}(-1)^{l+i}\text{ (change variable: }l=j-i)\\
 & = & \frac{1}{(k-i)!}(-1)^{i}(1+(-1))^{k-i}\\
 & = & 0.
\end{eqnarray*}
Therefore, 
\[
\sum_{j}U[i,j]U^{-1}[j,k]=\begin{cases}
1 & \text{ if }k=i\\
0 & \text{ if }k\ne i.
\end{cases}
\]
In other words, $U^{-1}$ can be defined as in Eq. (\ref{heq3}).

Now we establish the formula for $L$. By definition $A=LU$, so that $L[k,j]$ can be written as follows
\begin{eqnarray*}
L[k,j]=\sum_{i\le j}A_{n}[k,i]U^{-1}[i,j] & = & h^{j-k}\sum_{i=1}^{j}\frac{(n+i-1)!}{(n+i-k)!}\frac{(-1)^{i+j}}{((i-1)!(j-i)!)}.
\end{eqnarray*}

We now simplify the expression above. Consider the function $f(x)=x^{n}(1-x)^{j-1}$. On the one hand, we have 
\begin{align*}
f(x) & =x^{n}\sum_{i=0}^{j-1}\begin{pmatrix}j-1\\
i
\end{pmatrix}(-1)^{i}x^{i}\\
 & =\sum_{i=0}^{j-1}\begin{pmatrix}j-1\\
i
\end{pmatrix}(-1)^{i}x^{n+i}\\
 & =\sum_{i=1}^{j}\begin{pmatrix}j-1\\
i-1
\end{pmatrix}(-1)^{i-1}x^{n+i-1}.
\end{align*}
Therefore,
\begin{align*}
f(x)^{(k-1)} & =\sum_{i=1}^{j}\begin{pmatrix}j-1\\
i-1
\end{pmatrix}(-1)^{i-1}\frac{(n+i-1)!}{(n+i-1-(k-1))!}x^{n+i-1-(k-1)}\\
 & =\sum_{i=1}^{j}\frac{(j-1)!}{(j-i)!(i-1)!}(-1)^{i-1}\frac{(n+i-1)!}{(n+i-k)!}x^{n+i-1-(k-1)}.
\end{align*} 
It follows that
\begin{equation*}
\sum_{i=1}^{j}\frac{(-1)^{i}}{(i-1)!(j-i)!}\frac{(n+i-1)!}{(n+i-k)!}=\frac{-1}{(j-1)!}f(1)^{(k-1)},
\end{equation*}
and by multiplying both sides of this equality with $h^{j-k}$, we get
\begin{align*}
h^{j-k}\sum_{i=1}^{j}\frac{(-1)^{i+j}}{(i-1)!(j-i)!}\frac{(n+i-1)!}{(n+i-k)!} & =h^{j-k}\frac{(-1)^{j+1}}{(j-1)!}f(1)^{(k-1)}.
\end{align*}
On the other hand, according to Leibniz formula, we have
\begin{align*}
f(x)^{(k-1)} & =[x^{n}(1-x)^{j-1}]^{(k-1)}\\
 & =\sum_{i=0}^{k-1}\begin{pmatrix}k-1\\
i
\end{pmatrix}[x^{n}]^{(k-1-i)}[(1-x)^{j-1}]^{(i)}.
\end{align*}
If $k-1<j-1$, then $f(1)^{(k-1)}=0$, which implies that $\sum_{i\le j}A_{n}[k,i]U^{-1}[i,j]=0$.
\\ \ \\\
If $k\ge j$, then we have
\begin{align*}
f(x)^{(k-1)}|_{x=1} & =\sum_{i=0}^{k-1}\begin{pmatrix}k-1\\
i
\end{pmatrix}[x^{n}]^{(k-1-i)}[(1-x)^{j-1}]^{(i)}|_{x=1}\\
 & =\sum_{i=0}^{j-1}\begin{pmatrix}k-1\\
i
\end{pmatrix}[x^{n}]^{(k-1-i)}[(1-x)^{j-1}]^{(i)}|_{x=1}\\
 & +\sum_{i=j}^{k-1}\begin{pmatrix}k-1\\
i
\end{pmatrix}[x^{n}]^{(k-1-i)}[(1-x)^{j-1}]^{(i)}|_{x=1}\\
 & =\begin{pmatrix}k-1\\
j-1
\end{pmatrix}[x^{n}]^{(k-1-(j-1))}(-1)^{j-1}[(1-x)^{j-1}]^{(j-1)}|_{x=1}\\
 & \qquad+0\\
 & =\begin{pmatrix}k-1\\
j-1
\end{pmatrix}\frac{n!}{(n-(k-j)!)}x^{n-(k-j)}(j-1)!(-1)^{j-1}|_{x=1}\\
 & =\begin{pmatrix}k-1\\
j-1
\end{pmatrix}\frac{n!}{(n-(k-j))!}(j-1)!(-1)^{j-1}.
\end{align*}
Therefore, 
\begin{align*}
h^{j-k}\sum_{i=1}^{j}\frac{(-1)^{i+j}}{(i-1)!(j-i)!}\frac{(n+i-1)!}{(n+i-k)!} & =h^{j-k}(-1)^{j+1}(-1)^{j-1}\begin{pmatrix}k-1\\
j-1
\end{pmatrix}\frac{n!}{(n-(k-j))!}.
\end{align*}
It then follows that 
\begin{align*}
L(k,j)=\sum_{i\le j}A_{n}(k,i)U^{-1}(i,j) & =h^{j-k}\begin{pmatrix}k-1\\
j-1
\end{pmatrix}\frac{n!}{(n-k+j)!},
\end{align*}
which is the desired formula. This completes the proof of the first statement of the theorem.

Now we will prove the second statement of the theorem. We will show that $L^{-1}$ has the form as defined in (\ref{eq21}).

The element $(k,i)$ of the product of $L$ and $L^{-1}$ is given by $\sum\limits_{j=1}^{n}L[k,j]L^{-1}[j,i]$.

If $k<i$, then
\begin{equation}
\label{eq: L-1.1}
\sum_{j=1}^{n}L[k,j]L^{-1}[j,i]=\sum_{j=1}^{k}L[k,j]L^{-1}[j,i]=0.
\end{equation}
If $k>i$, then we have
\begin{align*}
\sum_{j=1}^{n}L[k,j]L^{-1}[j,i] & =\sum_{j=1,j\ge i,j\le k}^{n}L[k,j]L^{-1}[j,i]\\
 & =\sum_{j=i}^{k}h^{j-k}\begin{pmatrix}k-1\\
j-1
\end{pmatrix}\frac{n!}{(n-k+j)!}(-1)^{j-i}h^{i-j}\frac{(j-1)!}{(j-i)!(i-1)!}\frac{(n+j-i-1)!}{(n-1)!}\\
 & =h^{i-k}\sum_{j=i}^{k}\frac{(k-1)!}{(k-j)!}\frac{n!}{(n-k+j)!}(-1)^{j-i}\frac{1}{(j-i)!(i-1)!}\frac{(n+j-i-1)!}{(n-1)!}\\
 & =h^{i-k}\frac{(k-1)!n!}{(k-i)!(i-1)!(n-1)!}\sum_{j=0}^{k-i}\frac{(k-i)!}{(k-i-j)!j!}\frac{(n+j-1)!}{(n-k+i+j)!}(-1)^{j},
\end{align*}
where to obtain the last equality we have changed variable $j:=j+i$. Next, we will show that the summation in the above expression is equal to $0$.

Let $\alpha=k-i$ and consider $g(x)=x^{n-1}(x-1)^{\alpha}$.
On the on hand, we have
\begin{equation*}
g(x)=\sum_{j=0}^{\alpha}\begin{pmatrix}\alpha\\
j
\end{pmatrix}x^{n+j-1}(-1)^{\alpha-j},
\end{equation*}
which follows that
\begin{align*}
g^{(\alpha-1)}(x) & =\sum_{j=0}^{\alpha}\begin{pmatrix}\alpha\\
j
\end{pmatrix}[x^{n+j-1}]^{(\alpha-1)}(-1)^{\alpha-j}\\
 & =\sum_{j=0}^{\alpha}\begin{pmatrix}\alpha\\
j
\end{pmatrix}\frac{(n+j-1)!}{(n+j-1-(\alpha-1))!}x^{n+j-\alpha}(-1)^{\alpha-j}\\
 & =\sum_{j=0}^{\alpha}\begin{pmatrix}\alpha\\
j
\end{pmatrix}\frac{(n+j-1)!}{(n+j-\alpha)!}x^{n+j-\alpha}(-1)^{\alpha-j}\\
 & =(-1)^{\alpha}\sum_{j=0}^{k-i}\frac{(k-i)!}{(k-i-j)!j!}\frac{(n+j-1)!}{(n+j-k+i)!}x^{n+j-\alpha}(-1)^{j}.
\end{align*}
In particular, we obtain
\begin{equation*}
g^{(\alpha-1)}(1)=(-1)^{\alpha}\sum_{j=0}^{k-i}\frac{(k-i)!}{(k-i-j)!j!}\frac{(n+j-1)!}{(n+j-k+i)!}(-1)^{j}.
\end{equation*}
On the other hand, we have
\begin{align*}
g^{(\alpha-1)}(x) & =[x^{n-1}(x-1)^{\alpha}]^{(\alpha-1)}\\
 & =\sum_{j=0}^{\alpha-1}[x^{n-1}]^{(\alpha-1-j)}[(x-1)^{\alpha}]^{(j)},
\end{align*}
so $g^{(\alpha-1)}(1)=0$.

Therefore, 
\begin{equation*}
(-1)^{\alpha}\sum_{j=0}^{k-i}\frac{(k-i)!}{(k-i-j)!j!}\frac{(n+j-1)!}{(n+j-k+i)!}(-1)^{j}=0,
\end{equation*}
and hence we obtain for $k>i$,
\begin{equation}
\label{eq: L^-1.2}
\sum_{j=1}^{n}L[k,j]L^{-1}[j,i]=0.
\end{equation}
Finally when $k=i$, we have
\begin{align}
\sum_{j=1}^{n}L[k,j]L^{-1}[j,i] & =\sum_{j=1}^{i}L[k,j]L^{-1}[j,k]+\sum_{j=i+1}^{n}L[k,j]L^{-1}[j,k]\nonumber\\
 & =\sum_{j=1}^{k}L[k,j]L^{-1}[j,k]\nonumber\\
 & =L[k,k]L^{-1}[k,k]\nonumber\\
 & =1,\label{eq: L^-1.3}
\end{align}
where in the last step we have used that  $L[k,k]=L^{-1}[k,k]=1$.
From \eqref{eq: L-1.1}, \eqref{eq: L^-1.2}, and \eqref{eq: L^-1.3}, we conclude that \eqref{eq21} is indeed the formula for the inverse of $L$.

This finishes the proof of the theorem.
\end{proof} 
\section{Numerical investigations}
\label{sec: simulations}
In this section, we provide an algorithm to compute $C_{n,h}$ using results from the previous section and compute the expressions obtained for small $n$ explicitly.
\begin{algorithm}
\label{al: alg}
Algorithm to compute $C_{n,h}(\mathbf{x}, \mathbf{y})$:

Input: $n,h, \mathbf{x}=(x_0,\ldots,x_{n-1})\in \R^{dn}$ and $\mathbf{y}=(y_0,\ldots,y_{n-1})\in \R^{dn}$.

Output: $C_{n,h}(\mathbf{x}, \mathbf{y})$.

The algorithm consists of $5$ steps.

Step 1: Compute $B_{n}(h)$

\begin{equation*}
B_{n}(h)[i_{1},i_{2}]=\begin{cases}
(-1)^{n-i_{1}-1}\frac{(n+i_{2})!}{(i_{1}+i_{2}-n+1)!}h^{i_{2}+i_{1}-n+1}, & i_{2}+i_{1}\ge n-1,\\
0 & i_{2}+i_{1}<n-1.
\end{cases}
\end{equation*}

Step 2: Compute $\mathbf{b}_{n}(h)$
\[
\mathbf{b}_{n}(h)[i]=y_{i}-\sum_{j=i}^{n-1}\frac{1}{(j-i)!}h^{j-i}x_{j}.
\]

Step 3: Compute $L_{n}^{\text{\textminus}1}(h)$ and $U_{n}^{\text{\textminus}1}(h)$

\begin{equation*}
U^{-1}[i,j]=\begin{cases}
\frac{(-1)^{i+j}}{((i-1)!(j-i)!)h^{-j+i+n}} & \text{ if }j\ge i,\\
0 & \text{otherwise},
\end{cases},\quad L^{-1}[j,i]=\begin{cases}
(-1)^{j-i}h^{i-j}\frac{(j-1)!}{(i-1)!}\begin{pmatrix}n+j-i-1\\
j-i
\end{pmatrix} & \text{ if }j\ge i,\\
0 & \text{otherwise}.
\end{cases}
\end{equation*}

Step 4: Compute $A_{n}^{-1}=U_{n}^{-1}L_{n}^{-1}$.

Step 5: Compute $C$ using Theorem \ref{theo: explicit formula}

\begin{equation*}
C_{n,h}(x_{0},\ldots,x_{n-1};y_{0},\ldots,y_{n-1})=\mathbf{b}_{n}(h)^{T}B_{n}(h)[A_{n}(h)]^{-1}\mathbf{b}_{n}(h).
\end{equation*}

\end{algorithm}
Below we show $B_n(h),A_n(h), L_n(h), U_n(h), L_n^{-1}(h), U_n^{-1}(h)$ and $C_{n,h}$ for $n=1,2,3,4$ computed using Algorithm \ref{al: alg}. For simplicity of notation, we leave out the dependence on $h$ of the matrices.

\textbf{For $n=1$}.
\begin{align*}
& B=\left(
\begin{array}{c}
 1 \\
\end{array}
\right),
\\&A=\left(
\begin{array}{c}
 h \\
\end{array}
\right), \quad A^{-1}=\left(
\begin{array}{c}
 \frac{1}{h} \\
\end{array}
\right),
\\&L=\left(
\begin{array}{c}
 1 \\
\end{array}
\right),\quad U=\left(
\begin{array}{c}
 h \\
\end{array}
\right),
\\& L^{-1}=\left(
\begin{array}{c}
 1 \\
\end{array}
\right),\quad U^{-1}=\left(
\begin{array}{c}
 \frac{1}{h} \\
\end{array}
\right),
\\& C_{1,h}(x_0,y_0)=\frac{1}{h}\left(-x_0+y_0\right)^2.
\end{align*}
\textbf{For $n=2$}
\begin{align*}
&B=\left(
\begin{array}{cc}
 0 & -6 \\
 2 & 6 h \\
\end{array}
\right),
\\&A=\left(
\begin{array}{cc}
 h^2 & h^3 \\
 2 h & 3 h^2 \\
\end{array}
\right), \quad A^{-1}=\left(
\begin{array}{cc}
 \frac{3}{h^2} & -\frac{1}{h} \\
 -\frac{2}{h^3} & \frac{1}{h^2} \\
\end{array}
\right),
\\&
L=\left(
\begin{array}{cc}
 1 & 0 \\
 \frac{2}{h} & 1 \\
\end{array}
\right),\quad U=\left(
\begin{array}{cc}
 h^2 & h^3 \\
 0 & h^2 \\
\end{array}
\right),
\\&L^{-1}=\left(
\begin{array}{cc}
 1 & 0 \\
 -\frac{2}{h} & 1 \\
\end{array}
\right),\quad U^{-1}=\left(
\begin{array}{cc}
 \frac{1}{h^2} & -\frac{1}{h} \\
 0 & \frac{1}{h^2} \\
\end{array}
\right),
\\& C_{2,h}(x_0,x_1;y_0,y_1)=\frac{1}{h}\left[\left(y_1-x_1\right)^2+3 \left(y_1-x_1-\frac{2 \left(y_0-x_0-h x_1\right)}{h}\right)^2\right].
\end{align*}
\textbf{For $n=3$}
\begin{align*}
&B=\left(
\begin{array}{ccc}
 0 & 0 & 120 \\
 0 & -24 & -120 h \\
 6 & 24 h & 60 h^2 \\
\end{array}
\right),
\\&A=\left(
\begin{array}{ccc}
 h^3 & h^4 & h^5 \\
 3 h^2 & 4 h^3 & 5 h^4 \\
 6 h & 12 h^2 & 20 h^3 \\
\end{array}
\right),\quad A^{-1}=\left(
\begin{array}{ccc}
 \frac{10}{h^3} & -\frac{4}{h^2} & \frac{1}{2 h} \\
 -\frac{15}{h^4} & \frac{7}{h^3} & -\frac{1}{h^2} \\
 \frac{6}{h^5} & -\frac{3}{h^4} & \frac{1}{2 h^3} \\
\end{array}
\right)
\\&L=\left(
\begin{array}{ccc}
 1 & 0 & 0 \\
 \frac{3}{h} & 1 & 0 \\
 \frac{6}{h^2} & \frac{6}{h} & 1 \\
\end{array}
\right),\quad U=\left(
\begin{array}{ccc}
 h^3 & h^4 & h^5 \\
 0 & h^3 & 2 h^4 \\
 0 & 0 & 2 h^3 \\
\end{array}
\right),
\\&L^{-1}=\left(
\begin{array}{ccc}
 1 & 0 & 0 \\
 -\frac{3}{h} & 1 & 0 \\
 \frac{12}{h^2} & -\frac{6}{h} & 1 \\
\end{array}
\right),\quad U^{-1}=\left(
\begin{array}{ccc}
 \frac{1}{h^3} & -\frac{1}{h^2} & \frac{1}{2 h} \\
 0 & \frac{1}{h^3} & -\frac{1}{h^2} \\
 0 & 0 & \frac{1}{2 h^3} \\
\end{array}
\right),
\end{align*}
and
\begin{align*}
C_{3,h}(x_0,x_1,x_2;y_0,y_1,y_2)&=\frac{1}{h}\Bigg[\left(y_2-x_2\right)^2+3 \left(y_2-x_2-\frac{2 \left(y_1-x_1-h x_2\right)}{h}-\right)^2
\\&\quad+5 \left(y_2-x_2-\frac{6 \left(y_1-x_1-h x_2\right)}{h}+\frac{12 \left(y_0-x_0-h
x_1-\frac{h^2 x_2}{2}\right)}{h^2}\right)^2\Bigg].
\end{align*}
\textbf{For $n=4$}
\begin{align*}
&B=\left(
\begin{array}{cccc}
 0 & 0 & 0 & -5040 \\
 0 & 0 & 720 & 5040 h \\
 0 & -120 & -720 h & -2520 h^2 \\
 24 & 120 h & 360 h^2 & 840 h^3 \\
\end{array}
\right),
\\&A=\left(
\begin{array}{cccc}
 h^4 & h^5 & h^6 & h^7 \\
 4 h^3 & 5 h^4 & 6 h^5 & 7 h^6 \\
 12 h^2 & 20 h^3 & 30 h^4 & 42 h^5 \\
 24 h & 60 h^2 & 120 h^3 & 210 h^4 \\
\end{array}
\right),\quad A^{-1}=\left(
\begin{array}{cccc}
 \frac{35}{h^4} & -\frac{15}{h^3} & \frac{5}{2 h^2} & -\frac{1}{6 h} \\
 -\frac{84}{h^5} & \frac{39}{h^4} & -\frac{7}{h^3} & \frac{1}{2 h^2} \\
 \frac{70}{h^6} & -\frac{34}{h^5} & \frac{13}{2 h^4} & -\frac{1}{2 h^3} \\
 -\frac{20}{h^7} & \frac{10}{h^6} & -\frac{2}{h^5} & \frac{1}{6 h^4} \\
\end{array}
\right),
\\& L=\left(
\begin{array}{cccc}
 1 & 0 & 0 & 0 \\
 \frac{4}{h} & 1 & 0 & 0 \\
 \frac{12}{h^2} & \frac{8}{h} & 1 & 0 \\
 \frac{24}{h^3} & \frac{36}{h^2} & \frac{12}{h} & 1 \\
\end{array}
\right),\quad U=\left(
\begin{array}{cccc}
 h^4 & h^5 & h^6 & h^7 \\
 0 & h^4 & 2 h^5 & 3 h^6 \\
 0 & 0 & 2 h^4 & 6 h^5 \\
 0 & 0 & 0 & 6 h^4 \\
\end{array}
\right),
\\&L^{-1}=\left(
\begin{array}{cccc}
 1 & 0 & 0 & 0 \\
 -\frac{4}{h} & 1 & 0 & 0 \\
 \frac{20}{h^2} & -\frac{8}{h} & 1 & 0 \\
 -\frac{120}{h^3} & \frac{60}{h^2} & -\frac{12}{h} & 1 \\
\end{array}
\right), \quad U^{-1}=\left(
\begin{array}{cccc}
 \frac{1}{h^4} & -\frac{1}{h^3} & \frac{1}{2 h^2} & -\frac{1}{6 h} \\
 0 & \frac{1}{h^4} & -\frac{1}{h^3} & \frac{1}{2 h^2} \\
 0 & 0 & \frac{1}{2 h^4} & -\frac{1}{2 h^3} \\
 0 & 0 & 0 & \frac{1}{6 h^4} \\
\end{array}
\right),
\end{align*}
and
\begin{align*}
&C_{4,h}(x_0,x_1,x_2,x_3;y_0,y_1,y_2,y_3)
\\&\quad=\frac{1}{h}\Bigg[\left(y_3-x_3\right)^2+3 \left(y_3-x_3-\frac{2 \left(y_2-x_2-h x_3\right)}{h}\right)^2\Bigg]
\\&\quad+\frac{5}{h}\left(y_3-x_3-\frac{6 \left(y_2-x_2-h x_3\right)}{h}+\frac{12 \left(y_1-x_1-h x_2-\frac{h^2 x_3}{2}\right)}{h^2}\right)^2
\\&\quad+\frac{7}{h}\Bigg(y_3-x_3-\frac{12 \left(y_2-x_2-h x_3\right)}{h}+\frac{60 \left(y_1-x_1-h x_2-\frac{h^2 x_3}{2}\right)}{h^2}
\\&\qquad\qquad-\frac{120 \left(y_0-x_0-h x_1-\frac{h^2 x_2}{2}-\frac{h^3
x_3}{6}\right)}{h^3}\Bigg)^2.
\end{align*}
We observe that these cost functions satisfy a property that
\begin{align*}
C_{n,h}(x_0,x_1,\ldots, x_{n-1};y_0,y_1,\ldots,y_{n-1})&=C_{n-1,h}(x_1,\ldots, x_{n-1};y_1,\ldots,y_{n-1})\\&\qquad+\Big|\sum_{j=0}^{n-1}\frac{\alpha_{j}}{h^{n-1-j}}\Big[y_j-\sum_{i=j}^{n-1}\frac{h^{i-j}}{(i-j)!}x_i\Big]\Big|^2,
\end{align*}
for some $\{\alpha_j\}_{j=0}^{n-1}$, which is in accordance with Lemma \ref{theo: qual theo 2}.

\section*{Acknowledgements}
M. H. Duong was supported by ERC Starting Grant 335120. We would like to thank referees for their suggestions to improve the presentation of the paper.

\end{document}